\documentclass{article}
\usepackage[utf8]{inputenc}
\usepackage{amsmath}
\usepackage{amsfonts}
\usepackage{amssymb}
\usepackage{amsthm}
\usepackage{dsfont}

\newtheorem{theorem}{Theorem}[section]
\newtheorem{lemma}[theorem]{Lemma}

\newtheorem{remark}[theorem]{Remark}

\begin{document}
\begin{center}
\huge{A proof of the continuous martingale convergence theorem}
\end{center}
\begin{center}
\large{Joe Ghafari}
\end{center}

\begin{center}
\textbf{Abstract}
\end{center}

A proof of the continuous martingale convergence theorem is provided. It relies on a classical martingale inequality and the almost sure convergence of a uniformly bounded non-negative super-martingale, after a truncation argument.
\\*
\\*
\\*
\\
\\

\section{Introduction}
Unless otherwise stated, all processes in this paper are defined on a filtered probability space $\left(\Omega,\mathcal{F},\left(\mathcal{F}_r\right)_{r\in \mathbb{R}},\mathbb{P}\right).$
\\

The continuous martingale convergence theorem states the following:

\begin{theorem}
\label{th:convergence}
Let $\left(X_r\right)_{r \in \mathbb{R}}$ be a right-continuous integrable sub-martingale. 
\begin{enumerate}
\item If $\sup_{r \in \mathbb{R}_+}\mathbb{E}\left[\left|X_r\right|\right]<+\infty,$ then there exists an integrable random variable $X$ such that $\lim_{r \to +\infty} X_r=X$ a.s.
\item The limit $\lim_{r \to -\infty}X_r$ exists a.s. in $\left[-\infty,+\infty\right[.$
\end{enumerate}
\end{theorem}

The discrete version of this result has various proofs (see \cite{bib2}, \cite{bib4}, \cite{bib5}, \cite{bib7}, and \cite{bib11}). However, there doesn't seem to exist in the literature other
proofs than Doob's original proof of Theorem \ref{th:convergence}. The latter uses the up-crossing inequality (for more details, see \cite{bib8} and \cite{bib10}). 

In this paper we present a new elementary proof of theorem \ref{th:convergence}, avoiding the usual up-crossing lemma and using the classical inequality: 
\begin{equation}
\label{eq:I}
\delta \mathbb{P}\left(\sup_{r \in \mathbb{Q} \cap \left[u,u+v\right]}\left|Y_r\right|>\delta\right) \leq \mathbb{E}\left[\left|Y_{u}\right|\right]+2\mathbb{E}\left[\left|Y_{u+v}\right|\right], 
\end{equation}
where $(u,v,\delta) \in \mathbb{R} \times \left(\mathbb{R}_{+}\right)^2$ and $\left(Y_r\right)_{r \in \mathbb{R}}$ is an integrable super-martingale. 
\\*
The discrete martingale convergence theorem is also needed along with the following theorem due to Krickeberg.

\begin{theorem}[Krickerberg decomposition]
\label{th:Krickeberg}
Let $\left(Y_r\right)_{r \in \mathbb{R}_+}$ be an integrable sub-martingale for which $\sup_{r \in \mathbb{R}_+}\mathbb{E}\left[\left|Y_r\right|\right]<+\infty$. Then there exist a non-negative martingale $\left(U_{r}\right)_{r \in \mathbb{R}_+}$ and a non-negative super-martingale $\left(W_{r}\right)_{r \in \mathbb{R}_+}$ such that $\sup_{r \in \mathbb{R}_+}\mathbb{E}\left[\left|U_r\right|\right]<+\infty,\sup_{r \in \mathbb{R}_{+}}\mathbb{E}\left[\left|W_r\right|\right]<+\infty,$ and for all $r \in \mathbb{R}_+,Y_r = U_r - W_r.$ 
\end{theorem}

The reader is referred to \cite{bib8}, \cite{bib10}, and \cite{bib9} for a proof of inequality \eqref{eq:I} and Theorem \ref{th:Krickeberg}.

Lastly, we note that our approach gives a new proof of the theorem on the convergence of an integrable reversed discrete sub-martingale (see \cite{bib1} and \cite{bib6} for a general statement and a proof using the up-crossing lemma). 

\section{Preliminary tools}

In this section we present the necessary tools for our proof of Theorem \ref{th:convergence}.
\\

The following lemma is crucial for the proof of Theorem \ref{th:uniform} below.

\begin{lemma}
\label{lemma}
If $Y$ is an integrable random variable, then $\left(\mathbb{E}\left[Y|\mathcal{F}_{-k}\right]\right)_{k \in \mathbb{N}}$ converges to $\mathbb{E}\left[Y|\bigcap_{k \in \mathbb{N}}\mathcal{F}_{-k}\right]$ in $L^1.$
\end{lemma}
\begin{proof}
For every $k \in \mathbb{N},$ let $H_k:=\mathbb{E}\left[Y|\mathcal{F}_{-k}\right].$ We fix $\mathcal{F}_{\infty}:=\bigcap_{q \in \mathbb{N}}\mathcal{F}_{-q}.$ 

We consider first the case where $Y \in L^2.$ We have
\begin{equation*}
\forall (q,k) \in \mathbb{N}^2,\mathbb{E}\left[\left(H_{k+q}-H_q\right)^2\right]=\mathbb{E}\left[H^2_q\right]-\mathbb{E}\left[H^2_{k+q}\right].    
\end{equation*}
We note that the sequence $\left(\mathbb{E}\left[H_k^2\right]\right)_{k \in \mathbb{N}}$ is non-increasing and bounded below, hence it converges in $\mathbb{R}.$ 
\\*
Consequently, there exists $H \in L^2$ such that $\lim_{k \to +\infty}\mathbb{E}\left[\left(H_k-H\right)^2\right]=0$ and $H$ is $\mathcal{F}_{\infty}$-measurable.
\\*
Since
\begin{equation*}
\forall G \in \mathcal{F}_{\infty},\int_{G}\!H\,\mathrm{d}\mathbb{P}=\lim_{q \to +\infty}\int_{G}\!H_{q}\,\mathrm{d}\mathbb{P}=\int_{G}\!Y\,\mathrm{d}\mathbb{P},
\end{equation*}
it follows that $H=\mathbb{E}\left[Y|\mathcal{F}_{\infty}\right]$ a.s.

Turning to the general case, we have for all $(k,q) \in \mathbb{N}^2,$
\begin{equation*}
\mathbb{E}\left[\left|H_k-\mathbb{E}\left[Y|\mathcal{F}_{\infty}\right]\right|\right] \leq 2\mathbb{E}\left[\left|Y-Y \mathds{1}_{\left\{\left|Y\right| \leq q\right\}}\right|\right]+\mathbb{E}\left[\left|\mathbb{E}\left[Y \mathds{1}_{\left\{\left|Y\right| \leq q\right\}}|\mathcal{F}_{-k}\right]-\mathbb{E}\left[Y \mathds{1}_{\left\{\left|Y\right| \leq q\right\}}|\mathcal{F}_{\infty}\right]\right|\right].
\end{equation*}
Considering $k \to +\infty$ and from the special case which we already proved, we deduce that for every $q \in \mathbb{N},$ 
\begin{equation*}
\limsup_{k \to +\infty}\mathbb{E}\left[\left|H_k-\mathbb{E}\left[Y|\mathcal{F}_{\infty}\right]\right|\right] \leq 2\mathbb{E}\left[\left|Y-Y \mathds{1}_{\left\{\left|Y\right| \leq q\right\}}\right|\right].
\end{equation*}
Applying the dominated convergence theorem, we conclude the proof.
\end{proof}

We provide next a theorem on the almost sure convergence of a uniformly bounded non-negative super-martingale.

\begin{theorem}
\label{th:uniform}
If $\left(Y_r\right)_{r \in \mathbb{R}}$ is an integrable super-martingale such that 
\begin{equation*}
\exists c \in \mathbb{R}_+, \forall r \in \mathbb{R},0 \leq Y_r \leq c,
\end{equation*}
then there exist integrable random variables $Y$ and $Y'$ such that $\lim_{r \to +\infty, r \in \mathbb{Q}}Y_r=Y$ a.s. and $\lim_{r \to -\infty,r \in \mathbb{Q}}Y_r=Y'$ a.s.
\end{theorem}
\begin{proof}

\begin{itemize}
\item First, we will prove the almost sure convergence at $+\infty.$
\\*
Since $\sup_{k \in \mathbb{N}}\mathbb{E}\left[\left|Y_k\right|\right] \leq c<+\infty,$ it follows from the discrete martingale convergence theorem that $\left(Y_k\right)_{k \in \mathbb{N}}$ converges a.s. to a random variable $Y$ such that $0 \leq Y \leq c$ a.s. 
\\*
To prove the result in its generality, we begin by fixing $k \in \mathbb{N}^*$ and $\delta \in \mathbb{R}^{*}_{+}.$ 
\\*
Noticing that $\left(Y_{r+k}-Y_k\right)_{r \in \mathbb{R}_+}$ is a super-martingale relative to $\left(\mathcal{F}_{r+k}\right)_{r \in \mathbb{R}_+}$ and applying inequality \eqref{eq:I}, we obtain that for all $q \in \mathbb{N}^*,$
\begin{align*}
\delta \mathbb{P}\left(\sup_{r \in \mathbb{Q} \cap \left[0,q\right]}\left|Y_{r+k}-Y\right|>\delta\right) & \leq \delta \mathbb{P}\left(\sup_{r \in \mathbb{Q} \cap \left[0,q\right]}\left|Y_{r+k}-Y_k\right|>\frac{\delta}{2}\right)+ \delta \mathbb{P}\left(\left|Y_k-Y\right|>\frac{\delta}{2}\right) \\
& \leq 4\mathbb{E}\left[\left|Y_{q+k}-Y_k\right|\right]+2\mathbb{E}\left[\left|Y_k-Y\right|\right] \\ 
& \leq 4\mathbb{E}\left[\left|Y_{q+k}-Y\right|\right]+6\mathbb{E}\left[\left|Y_k-Y\right|\right].
\end{align*}
We note that $\lim_{q \to +\infty} \mathbb{E}\left[\left|Y_{q+k}-Y\right|\right]=0$ from the dominated convergence theorem.
\\*
Hence for any $k \in \mathbb{N}^*$ and all $\delta \in \mathbb{R}^{*}_{+},$
\begin{equation*}
\delta \mathbb{P}\left(\sup_{r \in \mathbb{Q} \cap \left[k,+\infty\right[}\left|Y_r-Y\right|> \delta\right)=\mathbb{P}\left(\bigcup_{q \in \mathbb{N}^*} \left\{\sup_{r \in \mathbb{Q}\cap \left[0,q\right]}\left|Y_{r+k}-Y\right|>\delta\right\}\right)\leq 6\mathbb{E}\left[\left|Y_k-Y\right|\right].
\end{equation*}
Taking $k \to +\infty$ and applying the dominated convergence theorem, we conclude that for all $\delta \in \mathbb{R}^{*}_{+},$
\begin{equation*}
\lim_{k \to +\infty}\mathbb{P}\left(\sup_{r \in \mathbb{Q} \cap \left[k,+\infty\right[}\left|Y_r-Y\right|>\delta\right)=0,
\end{equation*}
in other words $\lim_{r \to +\infty,r \in \mathbb{Q}}Y_r=Y$ a.s.
\item The proof of the almost sure convergence at $-\infty$ is essentially the same as before, therefore it's sufficient to show that the sequence $\left(Y_{-k}\right)_{k \in \mathbb{N}}$ converges a.s. 

We will provide an elegant proof that doesn't use the up-crossing inequality.

For every $k \in \mathbb{N},$ let $\Delta_k:=Y_{-k-1}-\mathbb{E}\left[Y_{-k}|\mathcal{F}_{-k-1}\right]$ and $V_k:=\sup_{q \in \mathbb{N}}\left(\sum_{n=0}^q\Delta_{k+n}\right).$
We note that the sequence $\left(\mathbb{E}\left[Y_{-k}\right]\right)_{k \in \mathbb{N}}$ is non-decreasing and bounded above, hence it has a finite limit denoted by $l.$
\\*
Also for every $k \in \mathbb{N}, 0 \leq V_k \leq V_{k+1}$ a.s., therefore by monotone convergence theorem we have
\begin{equation*}
\forall k \in \mathbb{N}, \mathbb{E}\left[V_k\right]=\sum_{q \in \mathbb{N}}\mathbb{E}\left[\Delta_{q+k}\right]=l-\mathbb{E}\left[Y_{-k}\right]<+\infty.     
\end{equation*}
So $\left(V_k\right)_{k \in \mathbb{N}}$ is a sequence of integrable random variable such that $\lim_{k \to +\infty} \mathbb{E}\left[V_k\right]=0.$

Next, we check that $\left(Y_{k}+V_{-k}\right)_{k \in \mathbb{Z}_{-}}$ is an integrable martingale relative to $\left(\mathcal{F}_{k}\right)_{k \in \mathbb{Z}_{-}}.$ 
\\*
We fix $k \in \mathbb{Z}_{-}. 
\\*
Y_{k}+V_{-k}$ is $\mathcal{F}_{k}$-measurable and integrable. We also have
\begin{align*}
\forall G \in \mathcal{F}_{k-1},\int_G\!Y_{k}\,\mathrm{d}\mathbb{P}+\int_G\!V_{-k}\,\mathrm{d}\mathbb{P} & =\int_G\!\mathbb{E}\left[Y_{k}|\mathcal{F}_{k-1}\right]\,\mathrm{d}\mathbb{P}+\sum_{q \in \mathbb{N}}\int_G\!\Delta_{q-k}\,\mathrm{d}\mathbb{P} \\
& =\int_G\!Y_{k-1}\,\mathrm{d}\mathbb{P}-\int_G\!\Delta_{-k}\,\mathrm{d}\mathbb{P}+ \sum_{q \in \mathbb{N}}\int_G\!\Delta_{q-k}\,\mathrm{d}\mathbb{P} \\
& =\int_G\!Y_{k-1}\,\mathrm{d}\mathbb{P}+ \sum_{q \in \mathbb{N}}\int_G\!\Delta_{q+1-k}\,\mathrm{d}\mathbb{P} \\
& =\int_G\!Y_{k-1}\,\mathrm{d}\mathbb{P}+\int_G\!V_{1-k}\,\mathrm{d}\mathbb{P}.
\end{align*}
Consequently, $\left(Y_{k}+V_{-k}\right)_{k \in \mathbb{Z}_{-}}$ is a martingale and hence for all $k \in \mathbb{N},Y_{-k}=\mathbb{E}\left[Y_{0}+V_0|\mathcal{F}_{-k}\right]-V_k$ a.s. 
\\*
It follows from Lemma \ref{lemma} that $\left(Y_{-k}\right)_{k \in \mathbb{N}}$ converges to $Y':=\mathbb{E}\left[Y_{0}+V_0|\bigcap_{k \in \mathbb{N}}\mathcal{F}_{-k}\right]$ in $L^1.$ 

Lastly, noticing that $\left(Y_{k}-Y'\right)_{k \in \mathbb{Z}_{-}}$ is a super-martingale relative to $\left(\mathcal{F}_{k}\right)_{k \in \mathbb{Z_{-}}}$ and using the discrete version of inequality \eqref{eq:I}, we have for every $\delta \in \mathbb{R}^{*}_{+}$ and every $(q,n) \in \mathbb{N}^2,$
\begin{equation*}
\delta \mathbb{P}\left(\max_{0 \leq k \leq q}\left|Y_{-k-n}-Y'\right|>\delta\right) \leq \mathbb{E}\left[\left|Y_{-q-n}-Y'\right|\right]+ 2\mathbb{E}\left[\left|Y_{-n}-Y'\right|\right].
\end{equation*}
Letting $q \to +\infty$ we obtain that for any $\delta \in \mathbb{R}^{*}_{+}$ and all $n \in \mathbb{N},$
\begin{equation*}
\mathbb{P}\left(\sup_{k \in \mathbb{N}}\left|Y_{-k-n}-Y'\right|>\delta\right) \leq \frac{2}{\delta}\mathbb{E}\left[\left|Y_{-n}-Y'\right|\right].
\end{equation*}
So for every $\delta \in \mathbb{R}^{*}_{+},\lim_{n \to +\infty}\mathbb{P} \left(\sup_{k \in \mathbb{N}}\left|Y_{-k-n}-Y'\right|>\delta \right)=0,$ concluding the proof.
\end{itemize}
\end{proof}

We end this section by stating and proving a general version of Theorem \ref{th:uniform}.

\begin{theorem}
\label{th:positif}
If $\left(Y_r\right)_{r \in \mathbb{R}}$ is a non-negative super-martingale, then the limits $\lim_{r \to +\infty,r \in \mathbb{Q}}Y_r$ and $\lim_{r\to -\infty,r \in \mathbb{Q}}Y_r$ exist a.s. in $\overline{\mathbb{R}}_+.$
\\*
In particular, if the sample paths of $(Y_r)_{r \in \mathbb{R}}$ are right-continuous, then $\lim_{r \to +\infty}Y_r$ and $\lim_{r \to -\infty}Y_r$ exist a.s. in $\overline{\mathbb{R}}_+.$
\end{theorem}
\begin{proof}
We will only prove the almost sure existence of the limit at $+\infty,$ the proof is analogous at $-\infty$.

The idea is to truncate properly so that the super-martingale property is preserved. 
\\*
We fix $q \in \mathbb{N}.$ 
\\*
$\left(\min\left(q,Y_r\right)\right)_{r \in \mathbb{R}}$ is a non-negative super-martingale uniformly bounded by $q.$ We deduce from Theorem \ref{th:uniform} that $\limsup_{r \to +\infty, r\in \mathbb{Q}}\min\left(q,Y_r\right)=\liminf_{r \to +\infty, r \in \mathbb{Q}}\min\left(q,Y_r\right)$ a.s. 
\\*
We also have the following relations: 
\begin{equation*}
\min\left(q,\limsup_{r \to +\infty, r \in \mathbb{Q}}Y_r\right)=\limsup_{r \to+ \infty, r \in \mathbb{Q}}\min\left(q,Y_r\right),    
\end{equation*}
\begin{equation*}
\min\left(q, \liminf_{r \to +\infty, r \in \mathbb{Q}}Y_r\right)=\liminf_{r \to +\infty, r \in \mathbb{Q}}\min\left(q,Y_r\right).
\end{equation*}
So for every $q \in \mathbb{N},\min\left(q,\limsup_{r \to +\infty, r \in \mathbb{Q}}Y_r\right) = \min\left(q, \liminf_{r \to +\infty, r \in \mathbb{Q}}Y_r\right)$ a.s. Hence $\limsup_{r \to +\infty, r \in \mathbb{Q}}Y_r=\liminf_{r \to +\infty, r \in \mathbb{Q}}Y_r$ a.s., yielding that $\lim_{r \to +\infty, r \in \mathbb{Q}}Y_r$ exists a.s. in $\overline{\mathbb{R}}_+.$

If $\left(Y_r\right)_{r \in \mathbb{R}}$ is right-continuous, then $\lim_{r \to +\infty}Y_r=\lim_{r \to +\infty, r \in \mathbb{Q}}Y_r$ a.s. 
\end{proof}

\begin{remark}
Theorem \ref{th:positif} holds true if $\mathbb{Q}$ is replaced by a countable dense subset $D$ of $\mathbb{R}.$
\end{remark}

\section{Proof of the continuous martingale convergence theorem}

Finally, we are ready to prove our theorem. 

\begin{proof}[Proof of Theorem \ref{th:convergence}]
Since the sample paths of $(X_r)_{r \in \mathbb{R}}$ are right-continuous, it's sufficient to prove that the limits of $\left(X_r\right)_{r \in \mathbb{Q}}$ at $+\infty$ and $-\infty$ exist almost surely.
\begin{enumerate}
    \item Applying Theorem \ref{th:Krickeberg}, there exist a non-negative martingale $\left(U_{r}\right)_{r \in \mathbb{R}_+}$ and a non-negative super-martingale $\left(W_{r}\right)_{r \in \mathbb{R}_+}$ such that $\sup_{r \in \mathbb{R}_+}\mathbb{E}\left[\left|U_r\right|\right]<+\infty, \sup_{r \in \mathbb{R}_{+}}\mathbb{E}\left[\left|W_r\right|\right]<+\infty,$ and for every $r \in \mathbb{R}_+, X_r=U_r-W_r.$
    \\*
    Theorem \ref{th:positif} yields that $\lim_{r \to +\infty,r \in \mathbb{Q}}U_r$ and $\lim_{r \to +\infty,r\in \mathbb{Q}}W_r$ exist a.s. in $\overline{\mathbb{R}}_+,$ we denote these almost sure limits by $U$ and $W,$ respectively. 
    \\*
    It follows by Fatou's lemma that $U$ and $W$ are integrable, in particular they are finite a.s. and hence $\lim_{r \to +\infty,r \in \mathbb{Q}}X_r=U-W \in L^1$ a.s.
    \item To verify the result, we need to write $X_r$ suitably. 
    \\*
    We note that $\left(\mathbb{E}\left[X_0^+|\mathcal{F}_r\right]\right)_{r \in \mathbb{R}_-}$ is an integrable martingale such that for all $r \in \mathbb{R}_-,\mathbb{E}\left[X_0^+|\mathcal{F}_r\right] \geq \mathbb{E}\left[X_0|\mathcal{F}_r\right] \geq X_r$ a.s., so $\left(\mathbb{E}\left[X_0^+|\mathcal{F}_r\right]-X_r\right)_{r \in \mathbb{R}_-}$ is a non-negative super-martingale. 
    \\*
    Applying again Theorem \ref{th:positif} and since $X_0^+ \in L^1,$ the limits $\lim_{r \to -\infty,r \in \mathbb{Q}}\mathbb{E}\left[X_0^+|\mathcal{F}_r\right]$ and $\lim_{r \to -\infty,r \in \mathbb{Q}}\left(\mathbb{E}\left[X_0^+|\mathcal{F}_r\right]-X_r\right)$ exist a.s. in $\mathbb{R}$ and $\overline{\mathbb{R}}_+,$ respectively. 
    \\*
    By writing for every $r \in \mathbb{R}_-, X_r=\mathbb{E}\left[X_0^+|\mathcal{F}_r\right]-\left(\mathbb{E}\left[X_0^+|\mathcal{F}_r\right]-X_r\right),$ we conclude that $\lim_{r \to -\infty,r \in \mathbb{Q}}X_r$ exists a.s. in $\left[-\infty,+\infty\right[.$  
\end{enumerate}
\end{proof}
\begin{remark}
Following the procedure used in proving Theorem \ref{th:convergence} we can show that if $u \in \mathbb{R}, D$ is a countable dense subset of $\mathbb{R},$ and $(Y_r)_{r \in \mathbb{R}}$ is an integrable sub-martingale, then there exist integrable random variables $V$ and $V'$ such that $\lim_{r \downarrow u, r \in D}Y_r=V$ a.s., $\lim_{r \uparrow u, r \in D}Y_r=V'$ a.s., $\mathbb{E}\left[V|\mathcal{F}_u\right] \geq Y_u$ a.s., and $\mathbb{E}\left[Y_u|\mathcal{F}_{u-}\right] \geq V'$ a.s. (see \cite{bib8} for another proof applying the up-crossing inequality). 
\end{remark}

\emph{Department of Mathematics and Statistics, University of Ottawa, Ottawa, K1N 6N5, Canada}

\textit{Email}: jghaf099@uottawa.ca

\end{document}